\documentclass[11pt, a4paper, reqno, english]{amsart}
\usepackage{amsmath} 
\usepackage{amsthm} 
\usepackage{amssymb} 
\usepackage[dvipsnames]{xcolor}
\usepackage{amscd} 
\usepackage{mathtools}

\usepackage{subcaption}

\usepackage{array} 
\usepackage{mlmodern}
\usepackage[cal=boondoxo,scr=euler]{mathalfa}
\usepackage[backref=page,linktocpage]{hyperref} 
\usepackage{cleveref} 
\usepackage{caption} %
\usepackage{graphics,graphicx} 
\usepackage{tikz,tikz-cd} 

\usetikzlibrary{graphs,graphs.standard,calc}

\usetikzlibrary{decorations.pathreplacing,}

\usepackage[shortlabels]{enumitem} 

\DeclareMathAlphabet{\mathsf}{OT1}{\sfdefault}{m}{n}

\newcommand{\nocontentsline}[3]{}
\newcommand{\tocless}[2]{\bgroup\let\addcontentsline=\nocontentsline#1{#2}\egroup}

\makeatletter
\@ifpackageloaded{newtxmath}{%
\renewcommand{\lambda}{\uplambda}%
\renewcommand{\zeta}{\upzeta}
\renewcommand{\xi}{\upxi}
\renewcommand{\tau}{\uptau}
\renewcommand{\alpha}{\upalpha}
}{
\usepackage{mlmodern}
}
\makeatother

\usepackage[margin=1.25in]{geometry}
\usepackage{verbatim}

\usepackage{scalerel}

\makeatletter
\def\dual#1{\expandafter\dual@aux#1\@nil}
\def\dual@aux#1/#2\@nil{\begin{tabular}{@{}c@{}}#1\\#2\end{tabular}}
\makeatother

\makeatletter
\@namedef{subjclassname@2020}{\textup{2020} Mathematics Subject Classification}
\makeatother

\DeclareMathAlphabet{\amathbb}{U}{bbold}{m}{n}

\hypersetup{
    colorlinks = true,
    linkbordercolor = {white},
    linkcolor = {BrickRed},
    anchorcolor = {black},
    citecolor = {BrickRed},
    filecolor = {cyan},
    menucolor = {BrickRed},
    runcolor = {cyan},
    urlcolor = {black}
}

\usetikzlibrary{automata}
\usepackage{bbm}

\newtheoremstyle{teoremas}
{11pt}
{11pt}
{\itshape}
{}
{\bfseries}
{}
{.5em}
{}

\theoremstyle{teoremas}
\newtheorem{theorem}{Theorem}[section]

\newtheorem{lemma}[theorem]{Lemma}
\newtheorem{proposition}[theorem]{Proposition}

\newtheoremstyle{definition}
{11pt}
{11pt}
{}
{}
{\bfseries}
{}
{.5em}
{}

\theoremstyle{definition}
\newtheorem{definition}[theorem]{Definition}

\newtheorem{problem}[theorem]{Problem}
\newtheorem{question}[theorem]{Question}
\newtheorem{example}[theorem]{Example}
\newtheorem{remark}[theorem]{Remark}

\crefname{theorem}{theorem}{theorems}
\Crefname{theorem}{Theorem}{Theorems}
\crefname{lemma}{lemma}{lemmas}
\Crefname{lemma}{Lemma}{Lemmas}
\crefname{proposition}{proposition}{propositions}
\Crefname{proposition}{Proposition}{Propositions}

\DeclareMathOperator{\rk}{rk}

\newcommand{\M}{\mathsf{M}}
\newcommand{\rev}{\operatorname{rev}}

\renewcommand{\H}{\mathrm{H}}

\newcommand{\Int}{\operatorname{Int}}

\AtBeginDocument{%
   \def\MR#1{}
}

\title{Eulerian posets and $Z$-polynomials}

\author[L.~Ferroni]{Luis~Ferroni}

\address{(L.~Ferroni)
  Dipartimento di Matematica, Universit\`a di Pisa, Pisa, Italy.
}
\email{luis.ferroni@unipi.it}

\author[R.~Riccardi]{Roberto Riccardi}

\address{(R.~Riccardi)
    Scuola Normale Superiore di Pisa, Pisa, Italy.
}
\email{roberto.riccardi@sns.it}

\thanks{The first part of this project was carried out when Luis Ferroni was a Minerva Research Fellow at the Institute for Advanced Study.}
\subjclass[2020]{Primary: 52B20, 05E99, 14N10, 14F43.}

\allowdisplaybreaks

\begin{document}

\begin{abstract}
Let $P$ be a finite partially ordered set. In a recent series of works, Proudfoot introduced the notion of $Z$-polynomials associated with $P$-kernels, providing a unified framework for various intersection cohomology Poincar\'e polynomials arising in diverse areas of mathematics. One of the problems posed by Proudfoot was to interpret the $Z$-polynomial in a fundamental setting---namely, when $P$ is the lattice of faces of a convex polytope (or, more generally, an Eulerian poset). We resolve this problem by proving that the $Z$-polynomial of any Eulerian poset coincides with the toric $h$-polynomial of the poset of all (possibly empty) closed intervals of $P$, ordered by reverse inclusion. Under suitable polyhedral conditions, this result identifies the $Z$-polynomial of a polytope with the Poincar\'e polynomial of the intersection cohomology of an associated auxiliary polytope. We prove some results about the Chow polynomials of the poset of intervals of an Eulerian poset and relate them to the Veronese transforms of polynomials.
\end{abstract}

\maketitle

\section{Introduction}

In his seminal work~\cite{stanley-local}, Stanley introduced a powerful framework to investigate polynomial invariants naturally associated with partially ordered sets and special incidence algebra elements called kernels. This construction provides a unified setting encompassing several, at first sight unrelated, theories that play a central role across mathematics. The three most representative instances are:
\begin{enumerate}[(i)]
\item The geometry and representation theory of Bruhat intervals of Weyl (or, more generally, Coxeter) groups \cite{kazhdan-lusztig,elias-williamson}, and references therein.
\item The enumeration of points, lines, planes, etc. in a hyperplane arrangement (or, more generally, a matroid) \cite{huh-wang,braden-huh-matherne-proudfoot-wang}.
\item The enumeration of (flags of) faces of rational convex polytopes \cite{stanley-toric-g,stanley-local} (or, more generally, arbitrary polytopes \cite{karu04}, regular CW-spheres and Cohen--Macaulay Eulerian posets \cite{karu}, or arbitrary posets \cite{bayer-ehrenborg}).

\end{enumerate}
This framework is often referred to as the Kazhdan--Lusztig--Stanley (KLS) theory of posets. For a comprehensive overview and an exposition of its algebro-geometric aspects, we refer the reader to the survey of Proudfoot~\cite{proudfoot-kls}. In what follows, we briefly recall the general setup of the KLS theory, adhering closely to the notation of~\cite{ferroni-matherne-vecchi}.

To each partially ordered set $P$ along with a $P$-kernel, one may associate two special incidence algebra elements called, respectively, the left and right \emph{Kazhdan--Lusztig--Stanley functions}. A third object, recently introduced by Proudfoot in the context of KLS theory \cite{proudfoot-kls}, is the so-called $Z$-function.

For the sake of having simpler statements, we define the \emph{$Z$-polynomial} of an Eulerian poset with minimum element $\widehat{0}$ and maximum element $\widehat{1}$ as the value that its $Z$-function assumes at the interval $[\widehat{0},\widehat{1}]$. Similarly, the \emph{right} (resp. \emph{left}) \emph{Kazhdan--Lusztig polynomials} are the polynomials that the right (resp. left) Kazhdan--Lusztig function associates to the interval $[\widehat{0},\widehat{1}]$.

In the three examples above, using adequate kernels, one has:
\begin{enumerate}[(i)]
    \item The right Kazhdan--Lusztig--Stanley polynomial of a Bruhat interval $[v,w]$ in a Weyl group $W$ (often just called the Kazhdan--Lusztig polynomial of $[v,w]$) encodes the local intersection cohomology Betti numbers of the Schubert variety $X_w$ at the point $e_v$ (see \cite[Corollary~4.8]{kazhdan-lusztig2}). The $Z$-polynomial, in turn, encodes the intersection cohomology Betti numbers of the Richardson variety $\overline{C}_v\cap \overline{X}_w$ (see \cite[Theorem~4.3]{proudfoot-kls}).
    \item The right Kazhdan--Lusztig--Stanley polynomial of (the lattice of flats of) a loopless matroid $\M$ encodes the Betti numbers of the stalk at the empty flat of the intersection cohomology module of $\M$ (see \cite{elias-proudfoot-wakefield,braden-huh-matherne-proudfoot-wang}). The $Z$-function encodes the Betti numbers of the intersection cohomology module of $\M$ (see \cite{proudfoot-xu-young,braden-huh-matherne-proudfoot-wang}).
    \item The left Kazhdan--Lusztig--Stanley polynomial of the face poset of a rational (resp. arbitrary) convex polytope, often called the \emph{toric $g$-polynomial}, encodes the Betti numbers of the (middle perversity) intersection cohomology of the toric variety associated to the polytope \cite{stanley-toric-g} (resp. the \emph{combinatorial} intersection cohomology \cite{barthel-brasselet-fieseler-kaup,bressler-lunts,karu04}).
\end{enumerate}

In the first and third examples above, the left KLS function equals the right KLS function of the dual poset, and both the class of Bruhat intervals and Eulerian posets are closed under duality. Therefore, there is no epistemological difference between left and right in these cases. In the second case, the left KLS function is trivial, whereas the right KLS function is very subtle. The absence of any mention of the $Z$-function in the third item above is the ultimate motivation for this paper. This problem was raised by Proudfoot in his survey \cite{proudfoot-kls} (see the last paragraph of Section~1.1 therein).

\begin{problem}[Proudfoot]
    Find an interpretation for the $Z$-function of the face poset of a convex polytope.
\end{problem}

Very little is known about $Z$-functions of polytopes, not even in small special cases. Unfortunately, the computation of $Z$-polynomials is a daunting task. Even for face posets of small-dimensional polytopes the defining recursion for the $Z$-polynomial is very slow in practice. This is a difficult obstruction to circumvent if one desires to build some intuition about what these polynomials are encoding.

Despite these computational obstacles, in this paper we answer this question in a strong sense. Instead of just focusing on polytopes, we view this problem in the broader setting of general Eulerian posets. Counterintuitively, even if the combinatorics of the posets is much more complicated (for example, they no longer are atomic and coatomic lattices), the answer to the above question becomes surprisingly beautiful and concise.

\begin{theorem}\label{thm:main}
    Let $P$ be an Eulerian poset. The $Z$-polynomial of $P$ using the Eulerian kernel equals the toric $h$-polynomial of the poset of intervals of $P$ ordered by reverse inclusion.
\end{theorem}

In the above statement and in several places of the present paper we treat the empty set as a closed interval of $P$, so that the poset alluded to in Theorem~\ref{thm:main} has a maximum (see Figure~\ref{fig:posets} below for an example). A non-obvious, but classical theorem due to Lindstr\"om (see \cite{lindstrom}) states that if $P$ is an Eulerian poset then the poset of all the intervals of $P$ ordered by reverse inclusion is Eulerian as well. The proof of this statement as well as the basics of posets of intervals are discussed in Section~\ref{sec:three}.

Recall that the \emph{toric $h$-polynomial} of an Eulerian poset $P$ is a ``symmetrization'' of the toric $g$-polynomial. When the poset $P$ is the lattice of faces of a simplicial polytope $\mathcal{P}$ the toric $h$-polynomial is the classical $h$-polynomial of $\mathcal{P}$, and when the polytope is rational but not necessarily simplicial, the coefficients of the toric $h$-polynomial are the (middle perversity) intersection cohomology Betti numbers of the toric variety associated to $\mathcal{P}$. Stanley \cite{stanley-toric-g} showed that the combinatorial recursion satisfied by the toric $h$-polynomial of the face poset of a polytope can be used to define it for arbitrary Eulerian posets (for a modern presentation, see \cite[p.~313]{stanley-ec1}). More generally, Bayer and Ehrenborg \cite{bayer-ehrenborg} defined toric $g$- and toric $h$-polynomials for arbitrary graded posets.

The above theorem fails if the poset is not Eulerian---as mentioned in the above paragraph, \cite{bayer-ehrenborg} states a definition of toric $h$- and toric $g$-polynomials for arbitrary graded posets, thus allowing for a natural definition of $Z$-polynomials of graded posets by convolving toric $g$-polynomials and toric $g$-polynomials of poset duals (see Section~\ref{sec:bayer-ehr} for the details).

There is a geometric way of visualizing Theorem~\ref{thm:main} when $P$ is the face poset of certain polytopes that we describe next. We say that a $d$-dimensional polytope $\mathcal{P}\subseteq \mathbb{R}^d$ is \emph{admissible} if the origin lies in its relative interior and, for every face $F$ of $\mathcal{P}$, the orthogonal projection of the origin to the affine hull of $F$ lies in the relative interior of $F$. If $\mathcal{P}$ is admissible, we construct the following auxiliary polytope $\mathcal{Q}\subseteq \mathbb{R}^{d+1}$:
    \[ \mathcal{Q} = \text{convex hull} \left\{(\mathcal{P}\times \{0\}) \cup (\mathcal{P}^* \times \{1\})\right\},\]
where $\mathcal{P}^*$ is the polar dual of the polytope $\mathcal{P}$.
A theorem due to Broadie \cite{broadie} shows that if $\mathcal{P}$ is admissible, the face poset of $\mathcal{Q}$ is precisely the poset of all the intervals of $P$. Moreover, Broadie also showed that if $\mathcal{P}$ is not admissible, the facial structure of $\mathcal{Q}$ differs from the one given by the intervals of $P$. When $\mathcal{P}$ is admissible, the polytope $\mathcal{Q}$ is often referred to as the \emph{geometric antiprism} of $\mathcal{P}$. From all this discussion, we have the following consequence of Theorem~\ref{thm:main}.

\begin{theorem}\label{thm:main-for-polytopes}
    Let $\mathcal{P}$ be an admissible polytope. Then, the $Z$-polynomial of the face poset of $\mathcal{P}$ encodes the intersection cohomology Betti numbers of the geometric antiprism of $\mathcal{P}$.
\end{theorem}

We learned from a personal communication with Nicholas Proudfoot that a version of the above statement was independently conjectured by Kalle Karu. 

Unfortunately, Theorem~\ref{thm:main-for-polytopes} cannot be extended to the case of general polytopes in any reasonable way. It is tempting to believe that if $P$ is the face poset of a convex polytope, then so is the poset of all intervals of $P$. Unfortunately, this belief turns out to be incorrect---we include a more detailed discussion about this in Section~\ref{sec:polytopes}.

\smallskip

When studying the poset of intervals of an Eulerian poset, it is natural to formulate questions about other invariants. Especially relevant in the context of KLS functions are the newly introduced \emph{Chow functions} defined in \cite{ferroni-matherne-vecchi}. They take their name due to the fact that in the case of hyperplane arrangements or matroids, they model the Hilbert--Poincar\'e series of a suitable Chow ring. 
In the Eulerian case, the Chow function encodes information about the number of chains of each size in the poset. 

We define the Chow polynomial of an Eulerian poset $P$ as the value that the Chow function assigns to the interval $[\widehat{0},\widehat{1}]$. We prove the following statement about the Chow polynomial of the poset of intervals of an Eulerian poset.

\begin{theorem}
    Let $P$ be an Eulerian poset of rank $r$. Then, the Chow polynomial of the poset of intervals of $P$ equals the second Veronese transform of the Chow polynomial of $P$. In particular, the Chow polynomial of $P$ determines the Chow polynomial of the poset of intervals of $P$ and vice versa.
\end{theorem}

(All the Chow polynomials alluded to in the last statement are \emph{Eulerian} Chow polynomials, i.e., Chow polynomials arising from the Eulerian kernel). For the precise definition of Veronese transforms, we refer the reader to the content of Section~\ref{sec:chow}. 

Finally, we include one last section in this article: Section~\ref{sec:five} contains several remarks and a few open questions that we consider relevant.

\section{Preliminaries}

\subsection{Incidence algebras and KLS functions}\label{sec:posetnotions}

Throughout this paper we will use the letter $P$ to denote a bounded partially ordered set. The minimum element of $P$ will be customarily denoted by $\widehat{0}$ whereas the maximum will be denoted by $\widehat{1}$. We say that $P$ is \emph{graded} if all the maximal chains from $\widehat{0}$ to $\widehat{1}$ in $P$ have the same length. This implies the existence of a map $\rho:P\to \mathbb{Z}_{\geq 0}$ called \emph{the rank function} of $P$, where $\rho(s)$ records $-1$ plus the size of any maximal chain starting at $\widehat{0}$ and ending at $s$. We will customarily write $\rho_{st}$ to denote $\rho(t) - \rho(s)$ for $s\leq t$.

The \emph{incidence algebra} of $P$, denoted by $\mathcal{I}(P)$, is the free $\mathbb{Z}[x]$-module spanned by $\Int(P)$, the set of all closed intervals of $P$. In other words, an element $a\in \mathcal{I}(P)$ associates to each closed interval $[s,t]\in \Int(P)$ a polynomial $a_{st}(x)\in \mathbb{Z}[x]$. When we need to specify the variable $x$, we will write $a_{st}(x)$. The product (also known as convolution) of two elements $a,b\in \mathcal{I}(P)$ is defined via
    \[ (ab)_{st} = \sum_{s\leq w\leq t} a_{sw}\, b_{wt} \qquad \text{ for every $s\leq t$ in $P$}.\]
This product operation makes $\mathcal{I}(P)$ into an associative algebra, having an identity $\delta$, given by $\delta_{st} = 0$ for $s < t$ and $\delta_{ss} = 1$ for all $s\in P$. The $\zeta$-function of $P$ is the element $\zeta\in \mathcal{I}(P)$ such that $\zeta_{st} = 1$ for all $s\leq t$. The M\"obius function of $P$ is the element $\mu =\zeta^{-1}$. 

We consider a subalgebra $\mathcal{I}_{\rho}(P)$ given by
    \[ \mathcal{I}_{\rho}(P) := \left\{ a \in \mathcal{I}(P) : \deg a_{st} \leq \rho_{st} \text{ for all $s\leq t$}\right\}.\]
There is an involution $\mathcal{I}_{\rho}(P) \to \mathcal{I}_{\rho}(P)$ that we denote by $a\mapsto a^{\rev}$, and is defined by the following equality:
    \[ (a^{\rev}_{st})(x) = x^{\rho_{st}} a_{st}(x^{-1}).\]
A $P$-kernel is an element $\kappa\in \mathcal{I}_{\rho}(P)$ such that $\kappa_{ss}=1$ for all $s\in P$ and $\kappa^{\rev} = \kappa^{-1}$. 

\begin{theorem}\label{thm:kls-functions}
    Let $\kappa\in \mathcal{I}_{\rho}(P)$ be a $P$-kernel. There exists a unique element $f\in \mathcal{I}_{\rho}(P)$ satisfying the following properties:
    \begin{enumerate}[\normalfont(i)]
        \item \label{it:f-i}$f_{ss}(x) = 1$ for all $s\in P$.
        \item \label{it:f-ii} $\deg f_{st}(x) < \frac{1}{2} \rho_{st}$ for all $s < t$.
        \item \label{it:f-iii} $f^{\rev} = \kappa\cdot f$.
    \end{enumerate}
    Similarly, there exists a unique element $g\in \mathcal{I}_{\rho}(P)$ satisfying the following properties:
    \begin{enumerate}[\normalfont(i')]
        \item \label{it:g-i} $g_{ss}(x) = 1$ for all $s\in P$.
        \item \label{it:g-ii} $\deg g_{st}(x) < \frac{1}{2} \rho_{st}$ for all $s < t$.
        \item \label{it:g-iii} $g^{\rev} = g \cdot \kappa$.
    \end{enumerate}
\end{theorem}

Following \cite[Section~2]{proudfoot-kls}, we refer to $f$ (respectively $g$) as the \emph{right} (respectively \emph{left}) Kazhdan--Lusztig--Stanley (KLS) function associated with $\kappa$. We further define
\[
f_P(x) := f_{\widehat{0}\,\widehat{1}}(x) \quad \text{and} \quad g_P(x) := g_{\widehat{0}\,\widehat{1}}(x),
\]
and call them the \emph{right} and \emph{left} \emph{Kazhdan--Lusztig--Stanley (KLS) polynomials} of $P$, respectively.

A complete proof of the preceding theorem can be found in \cite[Theorem~2.2]{proudfoot-kls}. Another fundamental object in the KLS theory is the \emph{$Z$-function}, extensively studied by Proudfoot in \cite{proudfoot-kls}.

\begin{definition}\label{def:zeta-functions}
Let $\kappa \in \mathcal{I}_{\rho}(P)$ be a $P$-kernel, and let $f$ (respectively $g$) denote its right (respectively left) KLS function.
The \emph{$Z$-function} associated with $\kappa$ is the element $Z \in \mathcal{I}_{\rho}(P)$ defined by
\[
Z := g^{\rev} f = g f^{\rev}.
\]
\end{definition}

The equality between these two expressions follows from the fact that $Z = g \kappa f$.
The $Z$-function is symmetric, meaning $Z^{\rev} = Z$, or equivalently,
\[
Z_{st}(x) = x^{\rho_{st}} Z_{st}(x^{-1})
\quad \text{for all } s \le t.
\]

\subsection{Eulerian Posets}

We will assume that the reader is acquainted with the basic properties of Eulerian posets, and we refer to \cite[Section~3.16]{stanley-ec1} for any undefined terminology. However, we make a brief recapitulation of some essential notions that we will use.

A poset $P$ is said to be \emph{Eulerian} if the M\"obius function satisfies $\mu_{st} = (-1)^{\rho(t) - \rho(s)}$ for every $s\leq t$ in $P$. An equivalent way of stating this property consists in saying that $P$ is Eulerian if and only if every nontrivial interval $[s,t]$ contains the same number of elements of odd rank and even rank. Famous examples of Eulerian posets include: face posets of convex polytopes, cell posets of regular CW-spheres, and Bruhat intervals of Coxeter groups.

The following provides a characterization of Eulerian posets in terms of kernels.

\begin{proposition}[{\cite[Proposition~7.1]{stanley-local}}]
    Let $P$ be a finite graded bounded poset. Then $P$ is Eulerian if and only if the element $\varepsilon\in \mathcal{I}_{\rho}(P)$ given by $\varepsilon_{st}(x) = (x-1)^{\rho_{st}}$ is a $P$-kernel.
\end{proposition}

We will refer to $\varepsilon$ as the \emph{Eulerian $P$-kernel}. The left KLS polynomial $g_P(x)$ arising from the $P$-kernel $\varepsilon$ in an Eulerian poset $P$ is called the \emph{toric $g$-polynomial of $P$}. It is not hard to see that the right KLS polynomial $f_P(x)$ arising from $\varepsilon$ equals the toric $g$-polynomial of the dual poset $P^*$. Finally, the \emph{toric $h$-polynomial} of $P$ is defined as
    \[ h_P(x) = \frac{g_P^{\rev}(x) - g_P(x)}{x-1}.\]
More generally, 
\[h_{st}(x) = \frac{g_{st}^{\rev}(x) -g_{st}(x)}{x-1}\] for each $s\leq t$ in $P$.
In Stanley's book \cite[Section~3.16]{stanley-ec1} this polynomial is denoted with the letter $f$, but recall that here we reserve that notation for the right KLS polynomial. The work of Bayer and Ehrenborg \cite{bayer-ehrenborg} provides a thorough study of toric $g$-polynomials and toric $h$-polynomials of more general (not necessarily Eulerian) posets. We recapitulate their general definition in Section~\ref{sec:bayer-ehr}.

\section{The poset of intervals}\label{sec:three}

Throughout the section we will assume that $P$ is an Eulerian poset. We set $\widehat{P}:= \Int(P) \cup \left\{ \varnothing \right\}$ to be the poset of all the closed intervals of $P$ ordered by reverse inclusion (in most sources, including \cite{stanley-ec1}, the empty set is ruled out as a closed interval, but in the present context we need it in order to obtain a bounded poset).

\begin{lemma} \label{lemma:rank-relation}
Let $P$ be an arbitrary graded bounded poset. The poset $\widehat{P}$ is graded, and for any pair of non-empty intervals $[s,t], [u,v] \in \Int(P)$ such that $[s,t] \subseteq [u,v]$, the rank function $\widehat{\rho}$ of $\widehat{P}$ satisfies:
\begin{equation}\label{eq:rank1}
    \widehat{\rho}_{[u,v],[s,t]}= \rho_{us}+ \rho_{tv}.
\end{equation}
In addition, we have
\begin{equation}\label{eq:rank2}
    \widehat{\rho}_{[s,t], \varnothing}= 1 + \rho_{st}
\end{equation}.
\end{lemma}

\begin{proof}
We first show that $\widehat{P}$ is graded. Pick any maximal chain $\mathcal{C}$ in $\widehat{P}$:
\[
    \mathcal{C} := \varnothing \subsetneq [s_1,t_1] \subsetneq [s_2,t_2] \subsetneq \cdots \subsetneq [s_n,t_n] \subsetneq [s_{n+1},t_{n+1}]=P.
\]
Since $\mathcal{C}$ is unrefinable, we have that $s_1=t_1=: q_1$, and for any $i \geq 1$ we have

\[
    [s_i,t_i] \subsetneq [s_{i+1},t_{i+1}] \Longrightarrow 
    \begin{cases}         s_i= s_{i+1} \text{ and } t_i \lessdot t_{i+1}  & \text{, or} \\ 
         t_i=t_{i+1} \text{ and } s_{i+1} \lessdot s_i \text{.}
        \end{cases}
\]
For each $i$, set $q_{i+1} := t_{i+1}$ if we are in the first case and $q_{i+1} := s_{i+1}$ if we are in the second one. By the definition of the $q_i$ it is not true in general that $q_{i}< q_{i+1}$, but they can be organized in an unrefinable chain of $P$. Thus, up to reordering the $q_i$'s, we have the following maximal chain in $P$:
\[
    \widehat{0}= q_1 < q_2 < q_3 < \cdots < q_n < q_{n+1}= \widehat{1}.
\]
Since $P$ is graded, it follows that $n = \rho(P)$, and thus any maximal chain in $\widehat{P}$ has length $\rho(P)+1$. This shows that $\widehat{P}$ is graded. This argument also proves the equality \eqref{eq:rank2} if we restrict to any interval $[s,t]$.

Now we prove (\ref{eq:rank1}).
    Given two maximal chains $u=u_0 < u_1 < \cdots < u_n=s$ and $t=t_0 < t_1 < \cdots < t_m=v$ with $n=\rho_{us}$ and $m=\rho_{tv}$, we have the following chain of strict inclusions of length $n+m$:
    \begin{equation*}
    [s,t] \subsetneq[u_{n-1},t] \subsetneq\cdots \subsetneq [u_1,t] \subsetneq[u,t] \subsetneq[u, t_1] \subsetneq\cdots \subsetneq [u,t_{m-1}] \subsetneq[u,v],
    \end{equation*}
    from which the inequality $\widehat{\rho}_{[u,v],[s,t]} \geq \rho_{us}+ \rho_{tv}$ follows.

    To prove the reverse inequality, using the property that $\widehat{\rho}_{[u,v],[s,t]} = \widehat{\rho}_{[u,v],[s,v]} + \widehat{\rho}_{[s,v],[s,t]} $, it suffices to show that $\widehat{\rho}_{[u,v],[s,v]} \leq \rho_{us}$ and $\widehat{\rho}_{[s,v],[s,t]} \leq \rho_{tv}$. For reasons of symmetry, we only deal with the second one, which is given by the fact that given a chain of strict inclusions $[s,t] \subsetneq[s,t_1] \subsetneq\cdots \subsetneq[s,v]$ in $\widehat{P}$ we obtain a chain of strict inequalities $t < t_1 < \cdots < v$ in $P$ of the same length as the chain in $\widehat{P}$. Hence we get  (\ref{eq:rank1}).
\end{proof}

It is a classical result of Lindstr\"om \cite{lindstrom} that if $P$ is Eulerian, then $\widehat{P}$ is also Eulerian. 

\begin{example}
    In Figure~\ref{fig:posets} we depict an Eulerian poset $P$ of rank $3$ and its corresponding poset of intervals $\widehat{P}$. A direct inspection on the M\"obius numbers of $\widehat{P}$ shows that it is Eulerian, as expected by Lindstr\"om's theorem. Furthermore, both $P$ and $\widehat{P}$ are Cohen--Macaulay (and hence Gorenstein*) --- the preservation of the Cohen--Macaulay property in this context was established by Baclawski in \cite[Theorem~7.2]{baclawski}.
    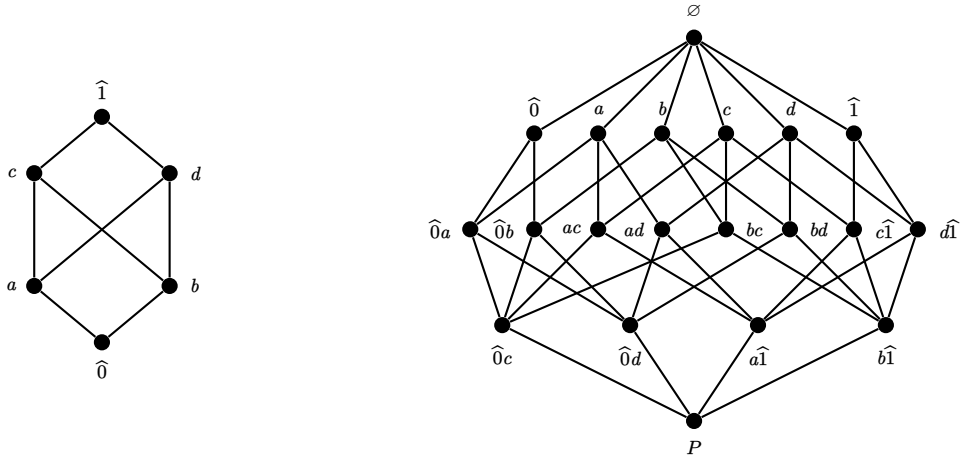
\begin{figure}[ht]
    \begin{tikzpicture}[baseline=(current bounding box.center),
  scale=0.75,auto=center,
  every node/.style={circle,scale=0.8,fill=black,inner sep=2.7pt}]
  \tikzstyle{edges} = [thick];

  \node[label=below:{\footnotesize \itshape $\widehat{0}$}] (cero) at (0,1) {};
  \node[label=left:{\footnotesize \itshape a}]    (a)   at (-1.2,2) {};
  \node[label=right:{\footnotesize \itshape b}]   (b)   at (1.2,2) {};
  \node[label=left:{\footnotesize \itshape c}]    (c)   at (-1.2,4) {};
  \node[label=right:{\footnotesize \itshape d}]   (d)   at (1.2,4) {};
  \node[label=above:{\footnotesize \itshape $\widehat{1}$}]  (uno)  at (0,5) {};

  \draw[edges] (a) -- (cero);
  \draw[edges] (b) -- (cero);
  \draw[edges] (a) -- (c);
  \draw[edges] (b) -- (c);
  \draw[edges] (a) -- (d);
  \draw[edges] (b) -- (d);
  \draw[edges] (d) -- (uno);
  \draw[edges] (c) -- (uno);
\end{tikzpicture}
\hspace{2.5cm}
\begin{tikzpicture}[baseline=(current bounding box.center),
  scale=0.85,auto=center,
  every node/.style={circle,scale=0.8,fill=black,inner sep=2.7pt}]
  \tikzstyle{edges}=[thick];

  \node[label=above:{\footnotesize \itshape $\varnothing$}] (n19) at (0,6) {}; 

  \node[label=above:{\footnotesize \itshape $\widehat{0}$}] (n0)  at (-2.5,4.5) {};  
  \node[label=above:{\footnotesize \itshape a}]   (n6)  at (-1.5,4.5) {};  
  \node[label=above:{\footnotesize \itshape b}]   (n10) at (-0.5,4.5) {};  
  \node[label=above:{\footnotesize \itshape c}]   (n14) at (0.5,4.5)  {};  
  \node[label=above:{\footnotesize \itshape d}]   (n16) at (1.5,4.5)  {};  
  \node[label=above:{\footnotesize \itshape $\widehat{1}$}]  (n18) at (2.5,4.5)  {};  

  \node[label=left:{\footnotesize \itshape $\widehat{0}$a}] (n1)  at (-3.5,3)  {}; 
  \node[label=left:{\footnotesize \itshape $\widehat{0}$b}] (n2)  at (-2.5,3)  {}; 
  \node[label=left:{\footnotesize \itshape ac}]   (n7)  at (-1.5,3)  {}; 
  \node[label=left:{\footnotesize \itshape ad}]   (n8)  at (-0.5,3)  {}; 
  \node[label=right:{\footnotesize \itshape bc}]   (n11) at (0.5,3)   {}; 
  \node[label=right:{\footnotesize \itshape bd}]   (n12) at (1.5,3)   {}; 
  \node[label=right:{\footnotesize \itshape c$\widehat{1}$}]  (n15) at (2.5,3)   {}; 
  \node[label=right:{\footnotesize \itshape d$\widehat{1}$}]  (n17) at (3.5,3)   {}; 

  \node[label=below:{\footnotesize \itshape $\widehat{0}$c}] (n3)  at (-3,1.5)  {}; 
  \node[label=below:{\footnotesize \itshape $\widehat{0}$d}] (n4)  at (-1,1.5)  {}; 
  \node[label=below:{\footnotesize \itshape a$\widehat{1}$}]  (n9)  at (1,1.5)   {}; 
  \node[label=below:{\footnotesize \itshape b$\widehat{1}$}]  (n13) at (3,1.5)   {}; 

  \node[label=below:{\footnotesize \itshape $P$}] (n5)  at (0,0) {};    

  \draw[edges] (n5) -- (n3);
  \draw[edges] (n5) -- (n4);
  \draw[edges] (n5) -- (n9);
  \draw[edges] (n5) -- (n13);

  \draw[edges] (n3) -- (n1);
  \draw[edges] (n3) -- (n2);
  \draw[edges] (n3) -- (n7);
  \draw[edges] (n3) -- (n11);

  \draw[edges] (n4) -- (n1);
  \draw[edges] (n4) -- (n2);
  \draw[edges] (n4) -- (n8);
  \draw[edges] (n4) -- (n12);

  \draw[edges] (n9)  -- (n7);
  \draw[edges] (n9)  -- (n8);
  \draw[edges] (n9)  -- (n15);
  \draw[edges] (n9)  -- (n17);

  \draw[edges] (n13) -- (n11);
  \draw[edges] (n13) -- (n12);
  \draw[edges] (n13) -- (n15);
  \draw[edges] (n13) -- (n17);

  \draw[edges] (n1)  -- (n0);
  \draw[edges] (n1)  -- (n6);
  \draw[edges] (n2)  -- (n0);
  \draw[edges] (n2)  -- (n10);
  \draw[edges] (n7)  -- (n6);
  \draw[edges] (n7)  -- (n14);
  \draw[edges] (n8)  -- (n6);
  \draw[edges] (n8)  -- (n16);
  \draw[edges] (n11) -- (n10);
  \draw[edges] (n11) -- (n14);
  \draw[edges] (n12) -- (n10);
  \draw[edges] (n12) -- (n16);
  \draw[edges] (n15) -- (n14);
  \draw[edges] (n15) -- (n18);
  \draw[edges] (n17) -- (n16);
  \draw[edges] (n17) -- (n18);

  \draw[edges] (n0)  -- (n19);
  \draw[edges] (n6)  -- (n19);
  \draw[edges] (n10) -- (n19);
  \draw[edges] (n14) -- (n19);
  \draw[edges] (n16) -- (n19);
  \draw[edges] (n18) -- (n19);
\end{tikzpicture}\caption{An Eulerian poset $P$ and its poset of intervals $\widehat{P}$.}\label{fig:posets}
\end{figure}
\end{example}

The following is a crucial result, which describes certain values of the left KLS function of $\widehat{P}$ in terms of the left and right KLS functions of $P$. From now on we denote by $\widehat{\rho}$ the rank function of $\widehat{P}$, and by $\widehat{\varepsilon}$ the Eulerian kernel in the incidence algebra of $\widehat{P}$. Similarly $\widehat{g}$ and $\widehat{h}$ will stand for the toric $g$- and toric $h$-functions of $\widehat{P}$.

\begin{theorem} \label{ghatformula}
    For any two non-empty intervals $[s,t] \subseteq [u,v]$ of $P$ we have the following formula:
    \begin{equation} \label{formulaghat}
        \widehat{g}_{[u,v],[s,t]}(x) = g_{us}(x) f_{tv}(x).
    \end{equation}
\end{theorem}

\begin{proof}
    Let us consider the incidence algebra element $z\in \mathcal{I}(\widehat{P})$ given by 
    \[z_{[u,v],[s,t]}(x)=\begin{cases} 
    g_{us}(x) f_{tv}(x) & \text{if $[s,t]\neq \varnothing$,}\\ \widehat{g}_{[u,v],\varnothing}(x) & \text{otherwise,}\end{cases}\]
    for each pair of closed intervals of $P$.
    
    By Theorem~\ref{thm:kls-functions} and the uniqueness of KLS functions, in order to show the validity of equation~\eqref{formulaghat}, it suffices to show that $z$ satisfies the following three properties:
    \begin{enumerate}[(i')]
        \item\label{it:first} $z_{[s,t],[s,t]} = 1$,
        \item\label{it:second} $\deg{z_{[u,v],[s,t]}(x)} < \frac{1}{2} \widehat{\rho}_{[u,v],[s,t]}$ for any $[s,t] \subsetneq [u,v]$, and
        \item\label{it:third} $z^{\rev} = z\cdot\widehat{\varepsilon}$ in $\mathcal{I}(\widehat{P})$.
    \end{enumerate}
    
    Note that \ref{it:first} is immediate from the fact that $f$ and $g$ are themselves the KLS functions in $P$, because $z_{[s,t],[s,t]} = g_{ss}(x)f_{tt}(x) = 1\cdot 1 = 1$.
    
    On the other hand, similarly, \ref{it:second} follows again from the fact that $f$ and $g$ are the KLS functions in $P$, along with the computation of the rank function of $\widehat{P}$:
    \[
       \deg z_{[u,v],[s,t]}(x) = \deg g_{us}(x) + \deg f_{tv}(x) < {\frac{1}{2}} ( \rho_{us} + \rho_{tv}) = \frac{1}{2}\widehat{\rho}_{[u,v],[s,t]}.
    \]

    We now prove \ref{it:third}:
    \begin{align*}
          z^{\text{rev}}_{[u,v],[s,t]}(x)   &=   x^{\widehat{\rho}_{[u,v],[s,t]}} \cdot z_{[u,v],[s,t]}(x^{-1}) \\
                &=  
         x^{\rho_{us}+ \rho_{tv}} g_{us}(x^{-1}) f_{tv}(x^{-1})  \\
               &= 
         x^{\rho_{us}} g_{us}(x^{-1}) \cdot x^{\rho_{tv}}f_{tv}(x^{-1}) \\
               &= 
         g^{\text{rev}}_{us}(x) f^{\text{rev}}_{tv}(x)  \\
               &= 
         (g \cdot \varepsilon)_{us} (x) ( \varepsilon \cdot f)_{tv}(x)  \intertext{where we applied the definitions of $z$ and the involution $\rev$, and where in the last step we used that $g$ is the left KLS function and $f$ is the right KLS function of $P$,}
               &=  
         \left(\sum_{u \leq w \leq s} {g_{uw}(x) (x-1)^{\rho_{ws}}}\right) 
        \left( \sum_{t \leq y \leq v} {(x-1)^{\rho_{ty}}f_{yv}(x)} \right)  \intertext{where we expanded the last product as a product of convolutions,}
               &= 
         \sum_{[s,t]\subseteq [w,y]\subseteq [u,v]} {g_{uw}(x) f_{yv}(x)} (x-1)^{\rho_{ws}+ \rho_{yv}}  \intertext{after rewriting the sum as a sum over intervals of $P$,}
               &= \sum_{[s,t]\subseteq [w,y]\subseteq [u,v]} {g_{uw}(x) f_{yv}(x)} (x-1)^{\widehat{\rho}_{[w,v],[s,y]}}\\ 
         &=( z \cdot \widehat{\varepsilon})_{[u,v],[s,t]}(x), 
    \end{align*}
    where in the last step we used the definition of $z$ once more. The proof is thus complete.
\end{proof}

Now, having the above formula for the toric $g$-function (that is, the left KLS function) of $\widehat{P}$, we can prove the following result, which directly implies the validity of Theorem~\ref{thm:main} when specialized to the case $[u,v] = [\widehat{0},\widehat{1}]$.

\begin{theorem}
    Let $P$ be an Eulerian poset, and let $\widehat{P}$ be the poset of closed intervals of $P$. Consider the $Z$-function, $Z\in \mathcal{I}(P)$ associated to the $\varepsilon$-kernel in $P$, and the toric $h$-function, $\widehat{h}\in \mathcal{I}(\widehat{P})$ associated to the kernel $\widehat{\varepsilon}$ in $\widehat{P}$. Then, for each pair of elements $u \leq v$ in $P$, we have:
    \begin{equation}
        Z_{uv}(x)= \widehat{h}_{[u,v],\varnothing}(x)
    \end{equation}
    In particular, the $Z$-function at any interval of $P$ equals the toric $h$-function at a principal upper order ideal of $\widehat{P}$.
\end{theorem}

\begin{proof}
    We have:
     \begin{align*}
          Z_{uv}(x)    &=     (g^{\text{rev}} \cdot f)_{uv}  \\
                &=  
          \sum_{u\leq t \leq v} { g^{\text{rev}}_{ut}(x) f_{tv}(x) },  \intertext{by definition of $Z$-function and convolution in $\mathcal{I}(P)$,}
                &=  
          \sum_{u\leq t \leq v} \left(  \sum_{u \leq s \leq t} g_{us}(x) (x-1)^{\rho_{st}}    \right) f_{tv}(x)  \intertext{because $g^{\rev} = g\cdot \varepsilon$ is the left KLS function associated to the kernel $\varepsilon$}
                &=  
          \sum_{\substack{[s,t]\subseteq [u,v]\\ [s,t]\neq \varnothing}} { g_{us}(x) f_{tv}(x) (x-1)^{\rho_{st}}   },\intertext{rewriting the sums in order to transform it to a convolution in $\mathcal{I}(\widehat{P})$ (note that the case $[s,t] = \varnothing$ may never happen)}
          &=
          \sum_{\substack{[s,t]\subseteq [u,v]\\ [s,t]\neq \varnothing}}  
          \widehat{g}_{[u,v],[s,t]}(x)\, 
          (x-1)^{\widehat{\rho}_{[s,t],\varnothing} - 1},     \intertext{where we employed Theorem~\ref{ghatformula} together with equation~\eqref{eq:rank2},}
          &= (x-1)^{-1} \sum_{\substack{[s,t]\subseteq [u,v]\\ [s,t]\neq \varnothing}}   { \widehat{g}_{[u,v],[s,t]}(x)\, (x-1)^{\widehat{\rho}_{[s,t],\varnothing}}   } \\
          &= (x-1)^{-1} \left(\left( \widehat{g} \cdot \widehat{\varepsilon}\right)_{[u,v],\varnothing}(x) - \widehat{g}_{[u,v],\varnothing}(x)\right), \intertext{by adding and subtracting the missing term to write the second factor using a convolution,}
          &= (x-1)^{-1}  \left(\widehat{g}^{\rev}_{[u,v],\varnothing}(x) - \widehat{g}_{[u,v],\varnothing}(x)\right),\intertext{by using that $\widehat{g}$ is the left KLS function associated to the kernel $\widehat{\varepsilon}$ in $\widehat{P}$}
          &= \widehat{h}_{[u,v],\varnothing}(x),
     \end{align*}
     where we used the definition of toric $h$-function in $\widehat{P}$. The proof is complete. 
\end{proof}

\section{Chow polynomials and Veronese transforms}\label{sec:chow}

If $P$ is Eulerian, the main construction by Ferroni, Matherne, and Vecchi in \cite{ferroni-matherne-vecchi} associates to the kernel $\varepsilon$ a special element $\H\in \mathcal{I}(P)$ as follows. First, define $\overline{\varepsilon}\in \mathcal{I}_{\rho}(P)$ as
    \[ \overline{\varepsilon}_{st}(x) = \begin{cases} (x-1)^{\rho_{st} - 1} & s < t,\\ -1 & s = t,\end{cases}\]
and consider $\H := -\left(\overline{\varepsilon}\right)^{-1}$.
This element is called the \emph{Chow function} and enjoys many interesting properties. We refer to \cite{ferroni-matherne-vecchi} for the precise definitions.

In \cite[Theorem~5.4]{ferroni-matherne-vecchi} it is proved that the Chow polynomial of an Eulerian poset $P$ equals the $h$-polynomial of $\Delta(P)$, the order complex of $P$. As we will see now, the enumeration of chains of $P$ according to size (that is, the $f$-vector of $\Delta(P)$) fully determines the $f$-vector of $\Delta(\widehat{P})$. Therefore, the Chow polynomial of $P$ fully determines the Chow polynomial of $\widehat{P}$. We briefly introduce the notion of Veronese transforms. 

\begin{definition}
    Let $\{h_i\}_{i=0}^r$ be a sequence of real numbers, let $d\geq r$, and consider the generating function:
        \begin{equation} \label{eq:gen-fun}
        \sum_{n=0}^{\infty}a_n\, x^n = \frac{h_r\, x^r + h_{r-1}\,x^{r-1}+\cdots + h_0}{(1-x)^{d+1}}.
        \end{equation}
    The \emph{$m$-th Veronese transform} of the sequence $\{h_i\}_{i=0}^r$ (with respect to degree $d$) is the unique sequence $h'_0,\ldots,h'_{d}$ satisfying:
    \[ \sum_{n=0}^{\infty}a_{mn}\, x^n = \frac{h'_d\, x^{d} + h'_{d-1}\,x^{d-1}+\cdots + h'_0}{(1-x)^{d+1}}.\]
\end{definition}

It is a standard fact about generating functions that, as long as the numerator on the right-hand side in equation \eqref{eq:gen-fun} does not vanish at $x=1$, the sequence $\{a_n\}_{n\geq 0}$ is interpolated by a polynomial of degree exactly $d$. In other words, there exists a degree $d$ polynomial $p(x)\in \mathbb{R}[x]$ such that $a_n = p(n)$ for each $n\in \mathbb{Z}_{\geq 0}$. This observation implies that the sequence $\{h'_i\}_{i=0}^d$ is enough to reconstruct the sequence $\{h_i\}_{i=0}^r$.

\begin{theorem}
    Let $P$ be an Eulerian poset of rank $r$. Then, the Chow polynomial of $\widehat{P}$ equals the second Veronese transform of the Chow polynomial of $P$ (with respect to degree $d=r$). In particular, the Chow polynomial of $P$ determines the Chow polynomial of $\widehat{P}$ and vice versa.
\end{theorem}

\begin{proof}
    Recall that the zeta polynomial (not to be confused with the $Z$-polynomial of Proudfoot) of a poset $Q$ enumerates multichains in $Q$ according to their size (see~\cite[Section~3.12]{stanley-ec1}). We use the notation $\zeta_Q(n)$ for the zeta polynomial in order to avoid confusion with the $Z$-polynomial. Precisely, $\zeta_Q(n)$ enumerates multichains of elements in $Q$ of size $n-1$.
    
    Because of how $\widehat{P}$ is defined, $\zeta_{\widehat{P}}(n)$ enumerates multichains of intervals of $P$. These can have the form
        \[ [s_1,t_1] \subseteq [s_2,t_2] \subseteq \cdots \subseteq [s_{n-1},t_{n-1}].\]
    However, we allow some of the intervals to be the empty set. If $[s_1,t_1] = \cdots = [s_j,t_j] = \varnothing$ for $j\geq 0$, this is equivalent to enumerating multichains of $2(n-j-1)$ elements in $P$:
        \[ s_{n-1} \leq s_{n-2} \leq \cdots \leq s_{j+1} \leq t_{j+1} \leq \cdots \leq t_{n-2} \leq t_{n-1}.\]
    This gives the relationship:
    \begin{equation}
    \zeta_{\widehat{P}}(n) = \sum_{j=0}^{n-1} \zeta_P(2(n-j-1)+1) = \sum_{j=0}^{n-1} \zeta_P(2j+1).
    \end{equation}
    In \cite[Proof of Theorem~5.4]{ferroni-matherne-vecchi} it is proved that the following relation holds for any Eulerian poset $Q$:
    \begin{equation}\label{eq:zeta-poset-to-chow}  \sum_{n=0}^{\infty} \zeta_Q(n+1)\, x^{n} = \frac{\H_Q(x)}{(1-x)^{\rk(Q)+1}}.
    \end{equation}
    Applying this for $Q = \widehat{P}$ (which has rank $r+1 = \rk(P) + 1$) results in:
    \[ \frac{\H_{\widehat{P}}(x)}{(1-x)^{r+2}} = \sum_{n=0}^{\infty} \zeta_{\widehat{P}}(n+1)\, x^{n} = \sum_{n=0}^{\infty} \sum_{j=0}^{n} \zeta_P(2j+1)\, x^{n} = \frac{1}{1-x} \sum_{n=0}^{\infty} \zeta_P(2n+1)\, x^n.\]
    Equating both sides in the last equation and using that $\zeta_P(0) = 0$ yields:
    \[\sum_{n=0}^{\infty} \zeta_P(2n+1)\, x^n = \frac{\H_{\widehat{P}}(x)}{(1-x)^{r+1}}.\]
    By applying again equation~\eqref{eq:zeta-poset-to-chow} for $Q=P$, we have that $\H_{\widehat{P}}(x)$ is the second Veronese transform of $\H_{P}(x)$, as claimed.
\end{proof}

\begin{example}
    Let $P$ be the Boolean lattice of rank $r$. The Chow polynomial of $P$ is an Eulerian polynomial. The poset $\widehat{P}$ is the face lattice of the $(r+1)$-dimensional cross-polytope. The Chow polynomial of $\widehat{P}$ is the type B Eulerian polynomial. It is a well-known fact that the second Veronese transform of the Eulerian polynomials agrees with the type B Eulerian polynomials.
\end{example}

\begin{remark}
    A result of Hetyei \cite[Proposition~12]{hetyei} implies that the (deconed) order complex of the poset of intervals of $\widehat{P}$ is a Tchebysheff triangulation of the first kind of the suspension of the (deconed) order complex of $P$. In particular, by \cite[Theorem~6]{hetyei} one can write down, in a different language, the relationship that relates $\H_P(x)$ and $\H_{\widehat{P}}(x)$ using the so-called Tchebysheff polynomials of the first kind.
\end{remark}

\begin{remark}
    A long-standing conjecture by Brenti and Welker in \cite{brenti-welker} asserts (in an equivalent form) that the Chow polynomial of the face poset of a polytope is real-rooted. A stronger version, suggested in the work by Athanasiadis and collaborators \cite{athanasiadis-tzanaki,athanasiadis-kalampogia}, suggests that the Chow polynomial of a Gorenstein* poset is real-rooted. See the discussion in \cite[Section~5.2]{ferroni-matherne-vecchi}.
    
    By combining the preceding theorem with a result by Anwar and Nazir \cite[Theorem~4.1]{anwar-nazir}, we have that the nonnegativity\footnote{It is still not known whether for all Eulerian posets the Chow polynomial has nonnegative coefficients. See \cite[Question~5.5]{ferroni-matherne-vecchi}.} of the coefficients of $\H_P(x)$ automatically implies the nonnegativity and real-rootedness of $\H_{\widehat{P}}(x)$. In other words, it follows that if $P$ is Gorenstein* then $\H_{\widehat{P}}(x)$ is real-rooted. This provides some (modest) evidence in favor of \cite[Conjecture~1.8]{ferroni-matherne-vecchi}.
\end{remark}

\section{Final remarks and open questions}\label{sec:five}

\subsection{Polytopes and geometric antiprisms}\label{sec:polytopes}

Whether the poset of intervals of the face poset of a polytope is itself the face poset of a polytope was considered as an open question in Gr\"unbaum's book \cite[p.~66]{grunbaum}, and was explicitly proposed as a problem by Lindstr\"om in \cite{lindstrom-problem}. Only recently a counterexample was found by Dobbins \cite{dobbins}.

On the positive side, however, a result by Bj\"orner \cite[Theorem~6.1]{bjorner} ensures that if $P$ is the face poset of a convex polytope, then the poset of intervals of $P$ is the face poset of a regular CW-sphere whose boundary complex is polyhedral and shellable. In particular, proving a desirable property of toric $h$-polynomials of this kind of sphere yields desirable properties for $Z$-polynomials of polytopes.

It is a well-known fact that rational polytopes do not exhaust all combinatorial types of polytopes. However, within this special class of posets (that is, face posets of rational polytopes) it is not known whether the operation of taking the poset of intervals produces another element in the class. 

Furthermore, Theorem~\ref{thm:main-for-polytopes} does not preclude the possibility that for an arbitrary rational polytope $\mathcal{P}$ there exists a non-projective compact toric variety whose intersection cohomology Betti numbers yield the coefficients of the $Z$-polynomial of the face poset of $\mathcal{P}$. We leave this as an open question.

\begin{question}
    Let $P$ be the face poset of a rational polytope. Is the $Z$-polynomial of $P$ the intersection cohomology Poincar\'e polynomial of some (not necessarily projective) compact toric variety?
\end{question}

\subsection{A comparison with the matroid case}

Let us explain how our results compare with the relationship between the
Kazhdan--Lusztig polynomial and the \(Z\)-polynomial of a matroid (see \cite{elias-proudfoot-wakefield} and \cite{proudfoot-xu-young} for the precise definition of these polynomials). For a matroid \(\M\) with lattice of flats \(\mathcal{L}(\M)\), Proudfoot--Xu--Young \cite{proudfoot-xu-young} define
\[
   Z_{\M}(x)=\sum_{F\in \mathcal{L}(\M)} x^{\operatorname{rk}F} P_{\M/F}(x),
\]
where \(\M/F\) denotes the contraction of \(\M\) by the flat \(F\). Conversely,
\[
   P_\M(x)=\sum_{F\in \mathcal{L}(\M)}
   \mu(\varnothing,F)x^{\operatorname{rk}F}Z_{\M/F}(x).
\]
Thus, in the matroid setting, \(Z_{\M}\) is built from the Kazhdan--Lusztig
polynomials of all contractions of \(\M\), and \(P_{\M}\) can be recovered from the
corresponding \(Z\)-polynomials by M\"obius inversion. In other words, the $Z$-polynomial of a matroid does not depend on the $Z$-polynomials of all minors, but rather on the $Z$-polynomials of all \emph{contractions} (or upper intervals of the poset). Geometrically, when \(\M\)
is realizable, \(P_{\M}\) is the intersection cohomology polynomial of the
reciprocal plane, whereas \(Z_{\M}\) is the intersection cohomology polynomial of a
projective variety \(Y(V)\) known as the matroid Schubert variety. The strata of \(Y(V)\) are indexed by flats,
and the contribution of the stratum indexed by \(F\) is governed by the
Kazhdan--Lusztig polynomial of the contraction \(\M/F\).

For Eulerian posets, and in particular for face posets of polytopes, the exact
analogue is not a one-sided sum over ``upper intervals''. The reason is that the matroidal left KLS function is trivial, whereas the Eulerian left KLS function is the toric \(g\)-function.

From the definitions, we have
\[
   Z_{uv}(x)
   =
   \sum_{u\leq s\leq t\leq v}
   g_{us}(x)\, f_{tv}(x)\,(x-1)^{\rho_{st}}.
\]
This equation makes the role of the interval poset more transparent. Let
\(\widehat P=\operatorname{Int}(P)\cup\{\varnothing\}\), ordered by reverse
inclusion, and let \(\widehat g,\widehat h\) be its toric \(g\)- and \(h\)-functions.
For \(u\leq s\leq t\leq v\), Theorem~\ref{ghatformula} gives
\[
   \widehat g_{[u,v],[s,t]}(x)=g_{us}(x)f_{tv}(x).
\]

In matroids, flats index
the pieces appearing in the formula for \(Z_{\M}\), and the relevant local
polynomials are the Kazhdan--Lusztig polynomials of contractions. For Eulerian
posets, intervals \([s,t]\) play the role of flats, and the corresponding local
polynomial is the product
\(
   g_{us}(x)f_{tv}(x),
\)
i.e. the toric \(g\)-polynomial of a lower interval times the toric \(g\)-polynomial
of a dual upper interval. The fact that this product is exactly
\(\widehat g_{[u,v],[s,t]}\) is the polytopal replacement for the matroidal
``contraction'' term.

When \(P\) is the face poset of an admissible rational polytope \(\mathcal P\),
Broadie's theorem identifies \(\widehat P\) with the face poset of the geometric
antiprism \(\mathcal Q\). Hence the equality $Z_P(x) = h_{\widehat{P}}(x)$ says that \(Z_P\) is the intersection
cohomology Poincar\'e polynomial of the toric variety associated with
\(\mathcal Q\). The usual torus-orbit stratification of this toric variety is
indexed by the faces of \(\mathcal Q\), and therefore, under Broadie's
identification, by intervals \([s,t]\) of \(P\). The formula in Theorem~\ref{ghatformula} describes the
toric \(g\)-polynomial attached to the interval between two such strata. Consequently, there is a much better analogy between the matroid Kazhdan--Lusztig polynomial and the toric $g$-polynomial than the analogy between the $Z$-polynomial of a realizable matroid and the $Z$-polynomial of a (rational, admissible) polytope. One should not expect \(Z_P\) to be expressed only in terms of the single
polynomial \(g_P\); just as the matroid formula uses the Kazhdan--Lusztig
polynomials of all contractions, the polytopal formula uses the toric
\(g\)-polynomials of all relevant intervals and dual intervals.

\subsection{Failure for more general posets}\label{sec:bayer-ehr}

In a classical paper by Bayer and Ehrenborg \cite{bayer-cd}, a more general notion of toric $h$-polynomial is defined for posets that are not necessarily Eulerian. It is convenient to state their definition using the terminology of incidence algebras. They define two elements $g,h\in \mathcal{I}(P)$ for an arbitrary graded poset $P$ as follows. First, they set as base cases $h_{ss}(x) = g_{ss}(x) = 1$ for every $s\in P$, and then the following double recursion for the remaining intervals:
    \begin{align*}
        h_{st}(x) &= (x-1)^{\rho_{st}} + \sum_{s < w < t} g_{sw}\, (x-1)^{\rho_{wt} - 1}\\
        g_{st}(x) &= U_{\leq \lfloor \frac{\rho_{st}}{2}\rfloor}\left((1-x) h_{st}(x)\right)
    \end{align*}
where $U_{\leq m}$ denotes the linear operator that truncates a polynomial to degree $\leq m$. It is not hard to see that this gives two well-defined objects that match our versions of the polynomials for Eulerian posets. It may be tempting to define an analog of the right KLS function (which we denoted as $f$) by considering 
    \[ f_{st}(x) = g^*_{ts}(x)\]
where $g^*_{ts}$ denotes the Bayer--Ehrenborg toric $g$-polynomial of the interval $[t,s]$ in the dual poset $P^*$. Once this is done, a natural definition for the $Z$-function would be: 
    \[ Z_{st}(x) = \sum_{s\leq w\leq t} x^{\rho_{sw}}\, g_{sw}(x^{-1})\, f_{wt}(x).\]
This of course agrees with Proudfoot's definition when $P$ is Eulerian. Unfortunately, Theorem~\ref{thm:main} is not true for arbitrary graded posets using these definitions. A direct computation with small graded non-Eulerian posets quickly yields several counterexamples.

\subsection{A relation with the \textbf{cd}-index}

Every Eulerian poset can be associated with a special bivariate polynomial $\Phi\in \mathbb{Z}\left<c,d\right>$ in two non-commuting variables $c$ and $d$, often called the \emph{cd-index}. We refer to \cite{bayer-cd} for a detailed account of the $cd$-index of Eulerian posets. A result of Bayer and Ehrenborg \cite[Theorem~4.2]{bayer-ehrenborg} allows for the computation of the toric $h$-polynomial of an Eulerian poset $P$ as a specialization of the $cd$-index of $P$. Furthermore, by a result of Joji{\'c} \cite{jojic}, the $cd$-index of the poset of intervals of $P$ is itself a (nice) specialization of the $cd$-index of $P$. Combining these two results with our Theorem~\ref{thm:main} implies that one may write down an explicit (but fairly complicated) specialization of the $cd$-index of $P$ that retrieves the $Z$-polynomial.

\subsection{Some open questions}

If $P$ is a Gorenstein* poset, a subtle result of Karu about the $cd$-index \cite{karu} implies that the (Eulerian) Chow polynomial of $P$ is $\gamma$-positive (see \cite[Theorem~5.7]{ferroni-matherne-vecchi}). It is natural to ask if the $Z$-polynomial of a Gorenstein* poset enjoys the same property. Unfortunately, no. In fact, being Gorenstein* is not enough to even guarantee the nonnegativity of the $Z$-polynomial or the $h$-polynomial. For example, the poset in Figure~\ref{fig:posets2} is Gorenstein*, yet we have:
    \begin{align*}
        h_P(x) &= x^3 - x^2 - x + 1,\\
        Z_P(x) &= x^4 - 2x^2 + 1.
    \end{align*}
However, the situation appears to be drastically different when the poset $P$ is a lattice in addition to being Gorenstein*. In fact, a long-standing (and still open) conjecture by Stanley \cite[Conjecture~4.2(d)]{stanley-toric-g} postulates that the toric $g$-polynomial of a Gorenstein* lattice is nonnegative, which is equivalent to asserting that the toric $h$-polynomial of such posets is unimodal. In particular, since $P$ being a Gorenstein* lattice implies $\widehat{P}$ being a Gorenstein* lattice, this conjecture would imply that $Z_P(x)$ should also be unimodal for all Gorenstein* lattices. However, computer experiments led us to consider the following question.

\begin{question}
    Are the $Z$-polynomials of Gorenstein* lattices $\gamma$-positive?
\end{question}

We have one additional reason to formulate the above question. The analogue of this statement for $Z$-polynomials of loopless matroids was proved by Ferroni, Matherne, Stevens, and Vecchi in \cite[Theorem~4.7]{ferroni-matherne-stevens-vecchi}.

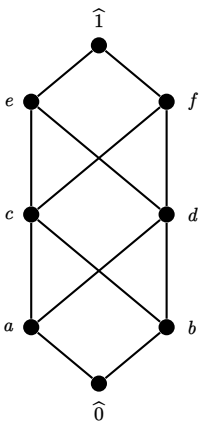
\begin{figure}[ht]
    \begin{tikzpicture}[baseline=(current bounding box.center),
  scale=0.75,auto=center,
  every node/.style={circle,scale=0.8,fill=black,inner sep=2.7pt}]
  \tikzstyle{edges} = [thick];
  \node[label=below:{\footnotesize \itshape $\widehat{0}$}] (cero) at (0,1) {};
  \node[label=left:{\footnotesize \itshape a}]    (a)   at (-1.2,2) {};
  \node[label=right:{\footnotesize \itshape b}]   (b)   at (1.2,2) {};
  \node[label=left:{\footnotesize \itshape c}]    (c)   at (-1.2,4) {};
  \node[label=right:{\footnotesize \itshape d}]   (d)   at (1.2,4) {};
  \node[label=left:{\footnotesize \itshape e}]    (e)   at (-1.2,6) {};
  \node[label=right:{\footnotesize \itshape f}]   (f)   at (1.2,6) {};
  \node[label=above:{\footnotesize \itshape $\widehat{1}$}]  (uno)  at (0,7) {};

  \draw[edges] (a) -- (cero);
  \draw[edges] (b) -- (cero);
  \draw[edges] (a) -- (c);
  \draw[edges] (b) -- (c);
  \draw[edges] (a) -- (d);
  \draw[edges] (b) -- (d);
  \draw[edges] (d) -- (e);
  \draw[edges] (d) -- (f);
  \draw[edges] (c) -- (e);
  \draw[edges] (c) -- (f);
  \draw[edges] (e) -- (uno);
  \draw[edges] (f) -- (uno);
\end{tikzpicture}\caption{A Gorenstein* poset with bad toric $h$- and $Z$-polynomials}\label{fig:posets2}
\end{figure}
Note that the example of Figure~\ref{fig:posets2} implies that an affirmative answer for the above question cannot simply come from the nonnegativity of the $cd$-index proved by Karu for Gorenstein* posets. In other words, even though it is possible to write a specialization of the $cd$-index that retrieves the $Z$-polynomial, this specialization \emph{cannot} be positive, and even less so for the $\gamma$-polynomial associated to $Z$.

\section*{Acknowledgments}

We thank Michele D'Adderio, Matt Larson, and Alan Stapledon for several conversations around the themes in this paper. We are especially grateful to Nick Proudfoot for motivating this research and for many suggestions and comments that improved the quality of this paper. Luis Ferroni is a member of the GNSAGA group of the Istituto Nazionale di Alta Matematica (INdAM).




\bibliographystyle{amsalpha}
\bibliography{bibliography}
\end{document}